\documentclass[12pt]{amsart}
\usepackage{amsfonts}
\usepackage{amsmath}
\usepackage[mathscr]{eucal}
\usepackage{eurosym}
\usepackage{amssymb}
\usepackage{amstext}
\usepackage{amsthm}
\usepackage{mathrsfs}
\usepackage{centernot}
\usepackage{enumitem}
\usepackage{graphicx, color}
\usepackage{caption}
\usepackage{subcaption}
\usepackage[a4paper, total={6.5in, 8in}]{geometry}
\usepackage{cool}
\usepackage{enumitem}

\AtBeginDocument{          }
\numberwithin{equation}{section}
\newtheorem{theorem}{Theorem}[section]

\newtheorem{corollary}[theorem]{Corollary}

\newtheorem{lemma}[theorem]{Lemma}

\newtheorem{proposition}[theorem]{Proposition}
\newtheorem{remark}[theorem]{Remark}

\begin{document}
\title[Discrete Bessel Functions]{Discrete Bessel functions and
discrete wave equation}
\author[A. Ba\v{s}i\'{c}]{Amar Ba\v{s}i\'{c}}
\address{Faculty of Electrical Engineering\\
University of Sarajevo\\
Zmaja od Bosne bb\\
71 000 Sarajevo\\
Bosnia and Herzegovina}
\email{abasic@etf.unsa.ba}
\author[L. Smajlovi\'{c}]{Lejla Smajlovi\'{c}}
\address{School of Economics and Business\\
University of Sarajevo\\
Trg Oslobodjenja Alija Izetbegovi\'{c} 1\\
71 000 Sarajevo\\
Bosnia and Herzegovina}
\email{lejla.smajlovic@efsa.unsa.ba}
\author[Z. \v{S}abanac]{Zenan \v{S}abanac}
\address{Department of Mathematics and Computer Science\\
University of Sarajevo\\
Zmaja od Bosne 35\\
71 000 Sarajevo\\
Bosnia and Herzegovina}
\email{zsabanac@pmf.unsa.ba}

\begin{abstract}
In this paper we study discrete Bessel functions which are solutions to the discretization of Bessel differential equations when the forward and the backward difference replace the time derivative. We focus on the discrete Bessel equations with the backward difference and derive their solutions.
We then study transformation properties of those functions, describe their asymptotic behaviour and compute Laplace transform.
As an application, we study the discrete wave equation on the integers in timescale $T=\mathbb{Z}$ and express its fundamental and general solution in terms of the discrete $J$-Bessel function. Going further, we show that the first fundamental solution of this equation oscillates with the exponentially decaying amplitude as time tends to infinity.
\end{abstract}

\maketitle

\vspace{1.5cc}
\section{Introduction and statement of results}

The classical $J$-Bessel and $I$-Bessel functions of the first kind are
important mathematical objects arising in different fields of mathematics and its applications.
They are defined for $z\in\mathbb{C}$ with $|\arg z|< \pi$, and a complex
index $\nu$ by the absolutely convergent series
\begin{equation}  \label{eq. J-Bessel def}
\mathcal{J}_{\nu}(z)= \left(\frac{z}{2}\right)^{\nu}\sum_{k=0}^{\infty}\frac{%
(-1)^k }{k! \Gamma(\nu+k+1)}\left(\frac{z}{2}\right)^{2k},
\end{equation}
and
\begin{equation}  \label{eq. I-Bessel def}
\mathcal{I}_{\nu}(z)= \left(\frac{z}{2}\right)^{\nu}\sum_{k=0}^{\infty}\frac{%
1 }{k! \Gamma(\nu+k+1)}\left(\frac{z}{2}\right)^{2k}.
\end{equation}

For a non-negative integer $\nu$, the $J$-Bessel and $I$-Bessel functions are
well defined by the series \eqref{eq. J-Bessel def} and \eqref{eq. I-Bessel
def} for all complex arguments $z$. When the index is a negative integer $-n$, then $\mathcal{J}_{-n}(z) = (-1)^n \mathcal{J}_n(z)$ and $\mathcal{I}_{-n}(z) = \mathcal{I}_n(z)$, for all complex arguments $z$.

For the purposes of this paper, it is important to notice that functions $\mathcal{J}%
_{\nu}(z) $ and $\mathcal{I}_{\nu}(z)$ are solutions to the Bessel differential equations
\begin{equation}  \label{eq. Bessel dif eq.}
z^2 \frac{d^2f}{dz^2}+ z \frac{df}{dz} - (\nu^2 - x^2)f=0 \quad \text{and}
\quad z^2 \frac{d^2f}{dz^2}+ z \frac{df}{dz} - (\nu^2 + x^2)f=0,
\end{equation}
respectively.

Moreover, the classical Bessel functions arise in solutions of diffusion and wave equations in different settings in which the solutions depend only on the distance between the spatial variables (the so-called radial dependence, see \cite{BM01}). For example, the solution to the wave equation on homogeneous trees, derived in \cite{CP94}, is expressed in terms of the $J$-Bessel function. In contrast, the solution to the diffusion equation on any $q$-regular graph deduced in \cite{CJK14}, is expressed in terms of the $I$-Bessel functions. (See also \cite{Slav22} for a more general diffusion equation and \cite{Slav17} for the wave equation).



\subsection{Discretizations of Bessel functions}

There exist many analogues and generalizations of Bessel functions. For
example, the $q-$Bessel functions are well studied (for an excellent introduction, see the thesis \cite{Sw92}), and their further generalizations to discrete timescales are also developed and applied in \cite{Ma13}, \cite{RKN19} or \cite{YYT22}, to name a few.

The starting point of this paper is discretizations of differential equations \eqref{eq. Bessel dif eq.} in which the forward or backward difference operator replaces the classical derivative. More precisely, we will study discretizations of Bessel functions that satisfy discrete analogues of the Bessel differential equations
\eqref{eq.
Bessel dif eq.} when the timescale equals $\mathbb{Z}$ and when the
derivative is either the forward or the backward difference. Those functions
will also arise as solutions of the diffusion/wave equation when the
timescale is $\mathbb{Z}$, the time derivative is the forward or the backward
difference and the spatial variable belongs to $\mathbb{Z}$.

The discretizations of equations \eqref{eq. Bessel dif eq.} are given for any integer $n$ and any nonzero
complex parameter $c$ are as follows:
\begin{itemize}
\item The forward difference equation
\begin{equation}
t\left( t-1\right) \partial _{t}^{2}y\left( t-2\right) +t\partial
_{t}y\left( t-1\right) \pm c^{2}t\left( t-1\right) y\left( t-2\right)
-n^{2}y\left( t\right) =0,  \label{modBesselEq}
\end{equation}
\item The backward difference equation
\begin{equation}
t\left( t+1\right) \overline{\partial }_{t}^{2}y_{n}\left( t+2\right) +t%
\overline{\partial }_{t}y_{n}\left( t+1\right) \pm c^{2}t\left( t+1\right)
y_{n}\left( t+2\right) -n^{2}y_{n}\left( t\right) =0.  \label{backBesselEq1}
\end{equation}
\end{itemize}

\noindent
Here $\partial_t$ denotes the \emph{forward difference operator} which acts on
functions $g$ defined on $\mathbb{Z}$ as
$$\partial_t g(t)=g(t+1)-g(t), \quad t\in\mathbb{Z}$$ while $
\overline{\partial}_t$ is the \emph{backward difference operator} acting as
$$
\overline{\partial}_t g(t) = g(t) - g(t-1), \quad t\in\mathbb{Z}.$$
Note that difference equations \eqref{modBesselEq} and \eqref{backBesselEq1} are different from equations in \cite{Boy61} where the discretized Laplace
equation in spherical coordinates was studied; see also \cite[example on p. 187]
{LL61} for a different type of discretization of equation \eqref{eq. Bessel dif eq.}.

The forward difference equation \eqref{modBesselEq} was first studied in
\cite{BC} with the plus sign in front of $c^2$, and for $c=1$. The detailed
study of this equation for a general $c\in\mathbb{C}\setminus\{0\}$ was carried out in \cite%
{Slav18}, where it was proved that the discrete $J$-Bessel function
\begin{equation}
J_{n}^{c}\left( t\right) =\frac{\left( -c/2\right) ^{n}\left( -t\right) _{n}%
}{n!}\Hypergeometric{2}{1}{\frac{n-t}{2},\frac{n-t}{2}+%
\frac{1}{2}}{n+1}{-c^2},\ t\in \mathbb{N}_0,\ n\in \mathbb{N}_{0},
\label{cJBessel}
\end{equation}%
and the discrete $I$-Bessel function
\begin{equation}
I_{n}^{c}\left( t\right) =\frac{\left( -c/2\right) ^{n}\left( -t\right) _{n}%
}{n!}\Hypergeometric{2}{1}{\frac{n-t}{2},\frac{n-t}{2}+\frac{1}{2}}{n+1}{c^2}%
,\ t\in \mathbb{N}_0,\ n\in \mathbb{N}_{0},  \label{cmodJBessel}
\end{equation}
are solutions to the forward difference equation \eqref{modBesselEq}, with
the plus and minus sign, respectively. Here, $_{2}F_{1}$ is the Gauss
hypergeometric function and $\left( t\right) _{k}$ is the Pochhammer symbol,
see equation \eqref{eq. hypergeom defn} below. Given that the Gauss
hypergeometric function $_{2}F_{1}(\alpha,\beta;\gamma;z )$ can be analytically continued into the complex $z$-plane cut along $[1,\infty]$, the function  $J_{n}^{c}\left( t\right)$ is well-defined for  $t\in\mathbb{Z}_{<0}$ and $c\in \mathbb{C}\setminus\{i\alpha:\, \alpha\in \mathbb{R},\, |\alpha|\geq 1\}$. Analogously, the function $I_{n}^{c}\left( t\right)$ is well defined for $t\in\mathbb{Z}_{<0}$ and $c\in \mathbb{C}\setminus\{\alpha:\, \alpha\in \mathbb{R},\, |\alpha|\geq 1\}$.

The first main result of this paper is the following theorem which describes two solutions to the backward difference equation \eqref{backBesselEq1}, thus providing definitions for two new discretizations of $J$-Bessel and $I$-Bessel functions, when the timescale is $\mathbb{Z}$, and the delta derivative is the backward difference.
\begin{theorem}
\label{thm: difference eq.} Let $c\in \mathbb{C}\setminus\{i\alpha:\, \alpha\in \mathbb{R},\, |\alpha|\geq 1\}$. Then, the function
\begin{equation}
\overline{J}_{n}^{c}\left( t\right) =\frac{\left( c/2\right) ^{n}\left(
t\right) _{n}}{n!}\Hypergeometric{2}{1}{\frac{n+t}{2},\frac{n+t}{2}+%
\frac{1}{2}}{n+1}{-c^2},\ t\in \mathbb{Z},\ n\in \mathbb{N}_{0},
\label{newcJBessel}
\end{equation}
is the solution to the backward difference equation \eqref{backBesselEq1},
with the plus sign.
\vskip .06in
\noindent
For $c\in \mathbb{C}\setminus\{\alpha:\, \alpha\in \mathbb{R},\, |\alpha|\geq 1\}$, the function
\begin{equation}
\overline{I}_{n}^{c}\left( t\right) =\frac{\left( c/2\right) ^{n}\left(
t\right) _{n}}{n!}\Hypergeometric{2}{1}{\frac{n+t}{2},\frac{n+t}{2}+%
\frac{1}{2}}{n+1}{c^2},\ t\in \mathbb{Z},\ n\in \mathbb{N}_{0},
\label{newcIBessel}
\end{equation}
is the solution to the backward difference equation \eqref{backBesselEq1},
with the minus sign.
\end{theorem}

We will call functions $J_n^c$ and $I_n^c$  \emph{%
forward} discrete Bessel functions, while functions $\overline{J}_{n}^{c}$
and $\overline{I}_{n}^{c}$ will be called \emph{backward} discrete Bessel
functions.

The complex parameter $c$ in the above theorem is chosen so that the hypergeometric function $\Hypergeometric{2}{1}{\alpha,\beta}{\gamma}{z}$ appearing in \eqref{newcJBessel} and \eqref{newcIBessel} is the principal branch of the analytic continuation of this function from the disc $|z|<1$ to the complex $z$-plane cut from 1 to $\infty$ along the real axes. Therefore, for any integer $t\in\mathbb{Z}$ functions $\overline{J}_{n}^{c}\left( t\right)$ and $\overline{I}_{n}^{c}\left( t\right)$ are holomorphic functions of $c$ in a given range. When $t\in\mathbb{Z}_{<0}$, the hypergeometric series in \eqref{newcJBessel} and \eqref{newcIBessel} is a polynomial in $c$, and hence holomorphic on the entire complex $c$-plane, see Proposition \ref{propJ} below.
\medskip

The forward discrete $J$-Bessel and $I$-Bessel functions have been studied in \cite{Slav18}, where many properties of those functions have been derived, including various transformation laws and analysis of sign changes for nonzero real $c$ and positive integers $t$.
An expression as a polynomial in $c$, a precise asymptotic behaviour
\begin{equation}\label{eq. I asympt as t to infty}
I_{n}^{c}\left( t\right) \sim \left( \mathrm{sgn}(c)\right) ^{n}%
\frac{\left( 1+\left\vert c\right\vert \right) ^{t+\frac{1}{2}}}{\sqrt{2\pi
t\left\vert c\right\vert }}\text{, as }t\rightarrow \infty,
\end{equation}
of $I_n^c(t)$, for $c\in\mathbb{R}\setminus\{0\}$ and a generating function for $I_n^c(t)$ was deduced in \cite[Section 3]{CHJSV}. A generating function for the backward discrete $I$-Bessel function $\overline{I}_{n}^{c}$ was derived in \cite{KS22} where some further properties of this function were established.

In this paper, we complete studies of analytic properties of the forward discrete $J$-Bessel function $J_n^c$ and the backward discrete $I$-Bessel function $\overline{I}_{n}^{c}$ and present a detailed study of properties of the backward discrete $J$-Bessel function $\overline{J}_{n}^{c}$, which is introduced in this paper.

We prove that $\overline{J}_{n}^{c}$ satisfies the properties analogous to the properties of the Bessel function $\mathcal{J}_n(t)$, see Lemma \ref{lemma2} below. Then, we proceed with the study of the
asymptotic behaviour of discrete Bessel functions as $t\to \infty$ which is summarized in the following theorem.

\begin{theorem}
\label{th_asymp} For any real, nonzero parameter $c$ and a fixed $n\in
\mathbb{N} $, we have that%
\begin{equation}
J_{n}^{c}\left( t\right) \sim \left( \mathrm{sgn}(c)\right) ^{n}\frac{\sqrt{2%
}}{\sqrt{\pi t\left\vert c\right\vert }}\left( 1+c^{2}\right) ^{\frac{t}{2}+%
\frac{1}{4}}\cos \left( \left( t+\frac{1}{2}\right) \theta -\frac{\pi }{4}+%
\frac{n\pi }{2}\right) \text{, as }t\rightarrow \infty ,  \label{asymp1}
\end{equation}

\begin{equation}
\overline{J}_{n}^{c}\left( t\right) \sim \left( \mathrm{sgn}(c)\right) ^{n}%
\frac{\sqrt{2}}{\sqrt{\pi t\left\vert c\right\vert }}\left( 1+c^{2}\right)
^{-\frac{t}{2}+\frac{1}{4}}\cos \left( \left( t-\frac{1}{2}\right) \theta -%
\frac{\pi }{4}+\frac{n\pi }{2}\right) \text{, as }t\rightarrow \infty ,
\label{asymp2}
\end{equation}
where $\mathrm{sgn}(c)$ denotes the sign of $c$ and $\theta\in\left(0,\frac{\pi}{2}\right)$ is such that $\cos \theta =\left( 1+c^{2}\right) ^{-\frac{1}{2}}$.

For any real, nonzero parameter $c$ with $|c|<1$ we have
\begin{equation}
\overline{I}_{n}^{c}\left( t\right) \sim \left( \mathrm{sgn}(c)\right) ^{n}%
\frac{\left( 1-\left\vert c\right\vert \right) ^{-t+\frac{1}{2}}}{\sqrt{2\pi
t\left\vert c\right\vert }}\text{, as }t\rightarrow \infty ,  \label{asymp3}
\end{equation}
\end{theorem}

Functions $J_{n}^{c}\left( t\right)$ and $I_{n}^{c}\left( t\right) $ are equal to zero when $n>t$, however, $\overline{J}_{n}^{c}\left( t\right)$ and $\overline{I}_{n}^{c}\left( t\right)$ are nonzero when $n>t$, hence it is of interest to deduce their asymptotic behaviour when $t$ is fixed and $n\to \infty$. This was carried out in Proposition \ref{asympJn} below.

\medskip

We also study the Laplace transform $\mathcal{L}_{\partial_t}$ associated  $\partial_t$ of functions $J_n^c$ and $I_n^c$  and the Laplace transform $\mathcal{L}_{\overline{\partial}_t}$ associated to $\overline{\partial}_t$ of functions $\overline{J}_{n}^{c}$ and $\overline{I}_{n}^{c}$ (precise definitions of $\mathcal{L}_{\partial_t}$ and $\mathcal{L}_{\overline{\partial}_t}$ are given in Section \ref{sec.
Laplace transf}) and prove the following theorem.

\begin{theorem}
\label{thm: Laplace tr} For $n\in \mathbb{N}_{0}$, $c\in \mathbb{C}%
\setminus \{0\}$ and any $z\in \mathbb{C}$, $z\neq \pm ic$, we have
\begin{equation}\label{eq. Lapl J}
\mathcal{L}_{\partial _{t}}\{J_{n}^{c}\}(z)=\mathcal{L}_{\overline{\partial }%
_{t}}\{\overline{J}_{n}^{c}\}(z)=\frac{c^{-n}\left( \sqrt{z^{2}+c^{2}}%
-z\right) ^{n}}{\sqrt{z^{2}+c^{2}}},
\end{equation}
while for $n\in \mathbb{N}_{0}$, $c\in \mathbb{C}%
\setminus \{0\}$ and any $z\in \mathbb{C}$, $z\neq \pm c$, we have
\begin{equation}\label{eq. Lapl I}
\mathcal{L}_{\partial _{t}}\{I_{n}^{c}\}(z)=\mathcal{L}_{\overline{\partial }%
_{t}}\{\overline{I}_{n}^{c}\}(z)=\frac{c ^{-n}\left(z- \sqrt{%
z^{2}-c^{2}}\right) ^{n}}{\sqrt{z^{2}-c^{2}}}.
\end{equation}
\end{theorem}

\begin{remark}\rm
According to \cite[formulas 17.13.103 and 17.13.109]{GR}, the right-hand side of \eqref{eq. Lapl J}
equals the classical Laplace transform of the Bessel function of the first
kind $\mathcal{J}_{n }\left( c x\right)$ at $z$, for $\mathrm{Re}(z)>|\mathrm{Im}(c)|$ while the right-hand side of \eqref{eq. Lapl I} equals the Laplace transform of the modified
Bessel function $\mathcal{I}_{n}\left( cx\right) $ at $z$, for $\mathrm{Re}(z)>|\mathrm{Re}(c)|$.

In the terminology of \cite[p. 1298]{DGJMR07}, this means that classical Bessel functions $\mathcal{J}_{n }\left( c x\right)$ and $\mathcal{I}_{n}\left( cx\right) $ viewed as functions of $x\in\mathbb{R}$ are \emph{shadow} functions for $J_{n}^{c}(t)$, $\overline{J}_{n}^{c}(t)$ and $I_{n}^{c}(t)$, $\overline{I}_{n}^{c}(t)$ respectively. Hence, those functions are indeed the \emph{appropriate timescale analogues} of classical Bessel functions.\footnote{We are thankful to Tom Cuchta for this remark.}
\end{remark}

\subsection{Discrete time wave equation on integers}

The classical wave equation in one dimension is the equation
\begin{equation*}
\frac{\partial ^{2}u(x;t)}{\partial ^{2}t}=c^{2}\frac{\partial ^{2}u(x;t)}{%
\partial ^{2}x},
\end{equation*}%
where $c>0$ is the propagation speed, $t\in \lbrack 0,\infty )$ is the time
variable, $x\in \mathbb{R}$ is a spatial variable and derivatives are
classical partial derivatives of real functions of two real variables.

In a more general setting, the wave equation on timescale $T$ on the spatial
space $X$ with the propagation speed $c>0$ can be viewed as the equation
\begin{equation}  \label{eq. wave general}
\Delta_t^2u(x;t) + c^2 \Delta_X u(x;t)=0,
\end{equation}
where $\Delta_X$ is usually the (weighted) Laplacian
on the spatial space $X$ (or some other generalization of the second derivative in
the space variable, e.g. fractional Laplacian, see \cite{G-CKLW19} or \cite{LM23}) and $\Delta_t$ stands for the delta (timescale) derivative with
respect to $t$ on a given timescale $T$, as described e.g. in \cite{BP1} and
\cite{BP2}. The initial conditions on $u$ and its derivative with respect to
$t$ for \eqref{eq. wave general} are defined in a natural way, depending on
the timescale.

When the timescale $T=[0,\infty)$ and $X$ is a homogeneous tree, the explicit
solution to \eqref{eq. wave general} with natural initial conditions was
deduced in \cite{MS99}; see also \cite{TS16}. In the special situation of $2$-regular tree (when $
X=\mathbb{Z}$) two independent fundamental solutions were given in terms of
the classical $J$-Bessel function \cite[Proposition 3]{MS99}, while the
asymptotic behaviour of the energy was studied in \cite{Me99}.

\medskip

When the timescale is discrete; more precisely when $T=\mathbb{Z}$, there
are different ways of discretizing the continuous time second derivative $%
\frac{\partial^2 }{\partial^2 t}$. For example, the forward difference $\partial_t$ and the
backward difference $\overline{\partial}_t$ are two natural delta derivatives
in time $t\in T=\mathbb{Z}$.\footnote{
The reason for a chosen notation for those derivatives stems from the fact
that $\overline{\partial}_t \partial_t =-\Delta_{\mathbb{Z}}$, where $%
\Delta_{\mathbb{Z}}$ is the combinatorial Laplacian on $\mathbb{Z}$. This is reminiscent
of the relation $\overline{\partial}_z\partial_z = -\frac{1}{4}
\Delta_{\mathbb{R}^2}$ between the Wirtinger derivatives $%
\partial_z= \frac{1}{2}\left( \frac{\partial}{\partial x} - i \frac{\partial%
}{\partial y}\right)$, $\overline{\partial}_z= \frac{1}{2}\left( \frac{%
\partial}{\partial x} + i \frac{\partial}{\partial y}\right)$  in $z=x+iy$ and the
Laplacian $\Delta_{\mathbb{R}^2}$ on $\mathbb{R}^2$.}

Therefore, the continuous time second derivative $\frac{\partial^2}{%
\partial^2 t}$ may be discretized as $\partial_t^2$ or $\overline{\partial}%
_t^2$, depending on whether one chooses the forward or the backward
difference operator as a discretization.

There are other possibilities; for example one may discretize $\frac{\partial^2}{%
\partial^2 t}$ as $\partial_t\overline{\partial}_t =
\overline{\partial}_t\partial_t= -\Delta_{\mathbb{Z}}$. Such a
discretization is studied in \cite{CP94}, where an explicit solution of the wave equation
on a homogeneous tree was found. The same timescale second derivative was used in \cite{AMPS13}, where the shifted wave equation on a homogeneous tree of degree $q+1>2$ is solved by applying a discrete
version of \'{A}sgeirsson's mean value theorem and by using the inverse dual
Abel transform that can be explicitly computed on the homogeneous tree.
A more general discrete wave
equation in which both operators are second-order differentials in different
timescales was studied in \cite[Section 3.2]{Ja06}.

In this paper, we will be interested in discrete analogues of the equation %
\eqref{eq. wave general} when the timescale $T=\mathbb{Z}$ and the spatial
space is $X=\mathbb{Z}$. The Laplacian on $X=\mathbb{Z}$ is the
combinatorial Laplacian on $X$ viewed as a 2-regular tree, in which every
vertex $n\in \mathbb{Z}$ is adjacent only to its neighbouring vertices $%
n-1,\,n+1$. In other words, in the setting of this paper, the action of the Laplacian $\Delta _{\mathbb{Z}}$ on any function $f:\mathbb{Z} \to \mathbb{R}$ is defined as
\begin{equation*}
\Delta _{X}f(n)=\Delta _{\mathbb{Z}}f(n)=2f(n)-f(n+1)-f(n-1),\quad n\in
\mathbb{Z}.
\end{equation*}%
The difference equation
\begin{equation}  \label{discwaveq1}
\partial _{t}^{2}u\left( n;t\right) +c^{2}\Delta _{\mathbb{Z}}u(n,t)=0,\,\, n\in
\mathbb{Z},\,\,t\in \mathbb{N}_{0},
\end{equation}%
with initial conditions
\begin{equation}  \label{dweqcon1}
u\left( n;0\right) =\left\{
\begin{array}{ll}
1 & \text{if }n=0, \\
0 & \text{if }n\neq 0,%
\end{array}%
\right. ,\quad \partial _{t}u\left( n;0\right) =0,\text{ \ }n\in \mathbb{Z},
\end{equation}%
was studied in \cite[Section 3]{Slav18}, where it is proved that the
function
\begin{equation}  \label{J2n}
u_{1}\left( n;t\right) =J_{2\left\vert n\right\vert }^{2c}\left( t\right)
\text{, \ }n\in \mathbb{Z}\text{, \ }t\in \mathbb{N}_{0}\text{,}
\end{equation}%
is its (fundamental) solution. In \cite[Section 3]{Slav18} Slav\' ik also deduced the solution to the wave equation \eqref{discwaveq1} with general initial conditions.

In Section \ref{wave_section} we study the analogue of the
equation \eqref{discwaveq1} with the derivative $\partial _{t}^{2}$
replaced by $\overline{\partial }_{t}^{2}$ with arbitrary initial
conditions, given by bounded real sequences indexed by integers. In Theorem \ref{thm. fund sol} below we prove that the first fundamental solution to \eqref{discwaveq1} with $\partial _{t}^{2}$ replaced by $\overline{\partial }_{t}^{2}$ subject to initial conditions \eqref{dweqcon1} is
$$
u_{1}\left( n;t\right) =\overline{J}_{2\left\vert n\right\vert }^{2c}\left( t\right)
\text{, \ }n\in \mathbb{Z}\text{, \ }t\in \mathbb{N}_{0}.
$$
We find the second fundamental solution and express the general solution as
a series involving $\overline{J}_{2\left\vert n\right\vert }^{2c}$ and the
initial data, see Theorem \ref{thgensol} for the exact statement.

Asymptotic behaviour of discretizations of $J$-Bessel functions proved in parts (i) and (ii) of Theorem \ref{th_asymp} yields asymptotic behaviour of fundamental solutions to \eqref{discwaveq1} subject to initial conditions \eqref{dweqcon1} with timescale derivatives being both the forward and the backward difference. In both cases solutions have oscillatory behaviour as $t\to \infty$, however the amplitude in the case when the time derivative is the forward difference grows exponentially with time, while the amplitude in the case when the time derivative is the backward difference decays exponentially with time; for precise statements, see Corollaries \ref{cor forward asympt} and \ref{cor backward asympt}.

This can be compared with the behaviour of the solution to \eqref{discwaveq1} when $\partial_t^2$ is replaced by $\partial_t \overline{\partial}_t$ studied in \cite{CP94} and which equals to zero when $t-n$ is odd.

\subsection{Organization of the paper}

The structure of the paper is the following: In Section \ref{sec. Prelim}, with the aim to make the paper self-contained, we provide a brief overview of necessary definitions and formulas related to the hypergeometric and Legendre functions and the Laplace transform for the forward and the backward difference. Section \ref{sec. Properties} is devoted to proving some properties of the backward discrete $J$-Bessel and $I$-Bessel functions and deducing their asymptotic behaviour for large $n$. Proofs solution to the discrete wave equation with initial conditions given in terms of arbitrary bounded real sequences indexed by integers. We end the paper by describing the asymptotic behaviour of fundamental solutions of the forward and backward discrete wave equations.

\vspace{1.5cc}
\section{Preliminaries}\label{sec. Prelim}

In this section we recall well known results and definitions that we will need in the sequel. More precisely, we define the hypergeometric function and the Legendre function of the first kind and recall some of their transformation properties and asymptotic behavior as certain parameters tend to infinity. In the last subsection we introduce the (unilateral) Laplace transform in timescales and express the Laplace transform in terms of a certain series when the timescale $T$ is the set of integers and the timescale (delta) derivative is the forward and the backward difference.

\setcounter{subsection}{0}
\subsection{Hypergeometric function}

In this subsection, we recall the basic properties of the Gauss hypergeometric
function defined for complex values of $z$ in the unit disc $|z|<1$ and parameters $\alpha ,\beta \in \mathbb{C}$, $\gamma \in \mathbb{C}\setminus \mathbb{N%
}_{0}$ as the absolutely convergent series
\begin{eqnarray}  \label{eq. hypergeom defn}
\Hypergeometric{2}{1}{\alpha,\beta}{\gamma}{z} &=&\sum\limits_{k=0}^{\infty }%
\frac{\left( \alpha \right) _{k}\left( \beta \right) _{k}}{\left( \gamma
\right) _{k}k!}z^{k}, \\
\left( t\right) _{k} &=&\frac{\Gamma \left( t+k\right) }{\Gamma \left(
t\right) }=\left\{
\begin{array}{ll}
t\left( t+1\right) \cdots \left( t+k-1\right) & \text{for }k\in \mathbb{N},
\\
1 & \text{for }k=0.%
\end{array}%
\right.  \notag
\end{eqnarray}

If $\alpha $ or $\beta $ is a nonpositive integer, then the series \eqref{eq. hypergeom defn} reduces to a finite sum
and converges everywhere. Otherwise, the hypergeometric series is convergent
in the unit disc $\left\vert z\right\vert <1$, but can be
analytically continued into the complex plane cut along $\left[ 1,\infty %
\right] $ (see e.g. \cite[Section 15.2]{OLBC} or \cite[Section 9.1]{Leb}). The
analytic continuation is denoted by the same symbol $%
\Hypergeometric{2}{1}{\alpha,\beta}{\gamma}{z}$.

For the sake of brevity and simplicity, we will use $F\left( \alpha ,\beta ;\gamma
;z\right) $ to denote the Gaussian hypergeometric function instead of $\Hypergeometric{2}{1}{\alpha,\beta}{%
\gamma}{z}$ .

It is obvious from series representation that $F\left( \alpha ,\beta ;\gamma;z\right)= F\left( \beta ,\alpha ;\gamma;z\right)$. We will need several recursion formulas
for Gauss hypergeometric function, which we recall from  \cite[Section 9.137]{GR} and \cite[formula 15.2.18]{AS} and list in the following lemma.
\begin{lemma}
\label{GHFrecurs} The Gauss hypergeometric function satisfies the following
recursion relations:
\begin{enumerate}[label=\normalfont(\roman*)]

\item $\gamma F\left( \alpha ,\beta ;\gamma;z\right) -\gamma F\left( \alpha
,\beta+1 ;\gamma;z\right) +\alpha zF\left( \alpha+1 ,\beta+1
;\gamma+1;z\right) =0$,

\item $\gamma F\left( \alpha ,\beta ;\gamma ;z\right) -\gamma F\left( \alpha
+1,\beta ;\gamma ;z\right)+\beta zF\left( \alpha +1,\beta +1;\gamma
+1;z\right)=0$,

\item $\gamma F\left( \alpha ,\beta ;\gamma ;z\right) -\left( \gamma -\alpha
\right) F\left( \alpha ,\beta ;\gamma +1;z\right) -\alpha F\left( \alpha+1
,\beta ;\gamma+1 ;z\right)=0$,

\item $\left( \gamma -\alpha -\beta \right) F\left( \alpha ,\beta ;\gamma
;z\right) -\left( \gamma -\alpha \right) F\left( \alpha -1,\beta ;\gamma
;z\right) \newline
+\beta \left( 1-z\right) F\left( \alpha ,\beta +1;\gamma ;z\right) =0$.
\end{enumerate}
\end{lemma}

The hypergeometric function satisfies many transformation formulas. In the sequel, we will use the following formula, which we quote from \cite[Section 9.131]{GR}:
\begin{equation}\label{eq. hypergom transf}
F(\alpha,\beta;\gamma;z) = (1-z)^{- \alpha}F\left(\alpha,\gamma-\beta;\gamma;\frac{z}{z-1}\right)
\end{equation}
and which is valid for $|\arg(1-z)|<\pi$.

\medskip

We will also need the asymptotic behaviour of $F(\alpha,\beta;\gamma;z)$ when some of the parameters are large. More precisely, we will make use of the following two asymptotic formulas which we quote from \cite[Section 9, p. 289]{Watson}. Let $\alpha$, $\beta$, $\gamma$ be arbitrary (fixed) complex numbers, $z=\cosh \zeta =\xi + i\eta \in \mathbb{C}\setminus (-\infty, 1]$ with $\xi, \, \eta$ real, $\xi\geq 0$, $-\pi< \eta\leq \pi$. The first formula states that
\begin{multline}\label{eq: Watson fla 1}
\left( \frac{z-1}{2}\right) ^{-\alpha -\lambda }F\left( \alpha +\lambda
,\alpha +\lambda -\gamma +1;\alpha -\beta +2\lambda +1;\frac{2}{1-z}\right)
\\
\sim \frac{2^{\alpha +\beta }\Gamma \left( \alpha -\beta +2\lambda +1\right)
}{\Gamma \left( \alpha +\lambda -\gamma +1\right) \Gamma \left( \gamma
-\beta +\lambda \right) }e^{-\left( \alpha +\lambda \right) \zeta }\left(
1-e^{-\zeta }\right) ^{\frac{1}{2}-\gamma }\left( 1+e^{-\zeta }\right)
^{\gamma -\alpha -\beta -\frac{1}{2}}\\
\times \sum\limits_{s=0}^{\infty }c_{s}^{\prime }\frac{\Gamma \left( s+\frac{%
1}{2}\right) }{\lambda ^{s+\frac{1}{2}}},
\end{multline}
as $|\lambda| \to +\infty$, where $\lambda$ is such that $|\arg \lambda| \leq \pi -\delta<\pi $. Constants $c_s'$ are independent of $\lambda$ and $c_0'=1$.
\vskip .06in
Assume $|\arg \lambda| \leq \pi/2 -\delta<\pi/2$. The second formula from \cite[Section 9]{Watson} states that
\begin{multline}\label{eq: Watson fla 2}
F\left( \alpha +\lambda ,\beta -\lambda ;\gamma ;\frac{1}{2}-\frac{1}{2}%
z\right) \\
\sim \frac{\Gamma \left( 1-\beta +\lambda \right) \Gamma \left( \gamma\right) }{%
\pi \Gamma \left( \gamma -\beta +\lambda \right) }2^{\alpha +\beta -1}\left(
1-e^{-\zeta }\right) ^{\frac{1}{2}-\gamma }\left( 1+e^{-\zeta }\right)
^{\gamma -\alpha -\beta -\frac{1}{2}} \\
\times \left[ e^{\left( \lambda -\beta \right) \zeta
}\sum\limits_{s=0}^{\infty }c_{s}\frac{\Gamma \left( s+\frac{1}{2}\right) }{%
\lambda ^{s+\frac{1}{2}}}+e^{\mp \pi i\left( \frac{1}{2}-\gamma \right)
}e^{-\left( \lambda +\alpha \right) \zeta }\sum\limits_{s=0}^{\infty
}c_{s}^{\prime }\frac{\Gamma \left( s+\frac{1}{2}\right) }{\lambda ^{s+\frac{%
1}{2}}}\right] ,
\end{multline}%
as $|\lambda| \to +\infty$, where $c_s$, $c_s'$ are independent of $\lambda$, $c_0=c_0'=1$ and in the second term the upper or lower sign is taken according as $\mathrm{Im}(z)\gtrless 0$.

\subsection{Legendre function of the first kind}

The Legendre function of the first kind of degree $\nu$ and order $\mu$, where $\nu,\mu\in\mathbb{C}$ is defined as
$$
P_{\mu}^{\nu}(z)= \frac{1}{\Gamma(1-\mu)}\left(\frac{z+1}{z-1}\right)^{\frac{\mu}{2}} F\left(-\nu,\nu+1;1-\mu;\frac{1-z}{2}\right).
$$

When $\beta=\alpha+1/2$, the hypergeometric function $F(\alpha,\beta;\gamma;z)$ can be expressed in terms of the Legendre function $P_{\mu}^{\nu}$. We quote here the formula 15.4.11 from \cite{AS}, which is valid for real, negative values of $z=x$ and any complex numbers $a,\,c$
\begin{equation}\label{eq. hypergeom = Leg}
F(a,a+1/2;c;x) = 2^{c-1}\Gamma(c)(-x)^{1/2-c/2} (1-x)^{c/2-a-1/2} P_{2a-c}^{1-c}\left[(1-x)^{-1/2}\right].
\end{equation}

When the order $\mu$ is an integer, say $m$, combining formulas 8.2.1. and 8.2.5. of \cite{AS}, we arrive at the following equation
\begin{equation}\label{eq. Leg at negatives}
  P_{-\nu-1}^{-m}(z)=P_{\nu}^{-m}(z)=\frac{\Gamma(\nu-m+1)}{\Gamma(\nu+m+1)}P_{\nu}^m(z).
\end{equation}

When $z=\cos\theta \in(0,1)$ in the sequel we will need the asymptotic formula for $P_{\nu}^{\mu}(\cos \theta)$, which we quote from \cite[formula 8.721.3]{GR} (see also \cite{Th})
\begin{equation}\label{eq. Leg asymptotics}
  P_{\nu}^{\mu}(\cos \theta)=\frac{2}{\sqrt{\pi}}\frac{\Gamma(\nu+m+1)}{\Gamma(\nu+3/2)}\frac{\cos\left( \left(\nu+\frac{1}{2}\right)\theta - \frac{\pi}{4} + \frac{\mu \pi}{2} \right)}{\sqrt{2 \sin \theta}}\left(1+O\left(\frac{1}{\nu}\right)\right),
\end{equation}
as $|\nu| \to\infty$.

\subsection{Laplace transform for the forward and the backward difference}\label{sec. Laplace transf}

The unilateral Laplace transform on timescales was introduced by Bohner and Peterson in \cite{BPLapl} and further used, generalized  and studied in numerous works.
In this section, we recall results from \cite{BG10} on Laplace transform on timescale $T=\mathbb{Z}$, with respect to both the forward difference operator $\partial_t$ and the backward difference operator $\overline{\partial}_t$. Given any timescale $T$ such that $0\in T$ and $\sup T=\infty$, the Laplace transform of the regulated function $x:T\to\mathbb{R}$ is defined by
$$
\mathcal{L}_{\Delta}\{x\}(z)=\int\limits_{0}^{\infty }\frac{x(t)}{%
e_{z}(t+\mu^{\ast}(t),0)}\Delta t,
$$
where $\Delta$ is the delta derivative on the timescale $T$; the integral is improper integral with respect to the derivative $\Delta$, $e_{z}(t,0)$ is the exponential function in timescale $T$ and $z$ belongs to a subset of complex numbers (depending on the function $x$) such that the above improper integral is convergent. The function $\mu^{\ast}(t)$ depends on the definition of the $\Delta$-derivative in timescale $T$.

The set of points $z$ for which the above integral converges is generally not easy to find. When $T=\mathbb{Z}$ and $\Delta$ is either the forward or the backward difference, the region of convergence will be a certain disc in the extended complex plane. We refer the interested reader to the paper \cite{JD21} in which a very detailed study of the Laplace transform, including the discussion on the region of convergence, is presented.


In this paper, we are interested only in the case when $T=\mathbb{Z}$ and $\Delta$ is either the forward or the backward difference. In our setup $\mu^{\ast}(t)=1$, for all $t$ when $\Delta=\partial_t$ and $\mu^{\ast}(t)=-1$ for all $t$ when $\Delta=\overline{\partial}_t$.

When $\Delta=\partial_t$, the exponential function is $e_{z}(t+1,0)=(1+z)^{t+1}$ (see, e.g. \cite{Cu15}, p. 29, formula (49)) and hence the Laplace transform of the sequence  $x(t):\mathbb{N}_0\to \mathbb{C}$ is given by
\begin{equation}
\mathcal{L}_{\partial _{t}}\{x\}(z)=\sum_{t=0}^{\infty }\frac{x(t)}{%
(1+z)^{t+1}},  \label{LaplaceFW}
\end{equation}
for all complex values of $z\neq -1$ such that the above series converges.
When $\Delta=\overline{\partial}_t$, the exponential function is given by $\hat{e}_{z}(t-1,0)=(1-z)^{-(t-1)}$, see \cite{Cu15}, p. 29, formula (50), hence the Laplace transform of the sequence  $x(t):\mathbb{N}_0\to \mathbb{C}$ in this case is
\begin{equation}
\mathcal{L}_{\overline{\partial }_{t}}\{x\}(z)=\sum_{t=0}^{\infty
}x(t)(1-z)^{t-1},  \label{LaplaceBW}
\end{equation}%
for all complex numbers $z\neq 1$ such that the above series converges.

\vspace{1.5cc}
\section{Properties of backward discrete Bessel Functions}
\label{sec. Properties}

In this section, we derive some basic properties of backward discrete $J$-Bessel and $I$-Bessel functions. We start by stating simple relations between four discrete Bessel functions that stem directly from their definitions \eqref{cJBessel} -- \eqref{newcIBessel}. Namely, we have the following identities which hold true for all $n\in\mathbb{N}_0$ and $t\in\mathbb{Z}$:
$$
\overline{J}_{n}^{c}\left( t\right)=(-1)^n J_n^c(-t), \quad \text{for}\quad c\in \mathbb{C}\setminus\{i\alpha:\, \alpha\in \mathbb{R},\, |\alpha|\geq 1\},
$$
$$
\overline{I}_{n}^{c}\left( t\right)=(-1)^n I_n^c(-t), \quad \text{for} \quad c\in \mathbb{C}\setminus\{\alpha:\, \alpha\in \mathbb{R},\, |\alpha|\geq 1\},
$$
and
$$
I_{n}^{c}\left( t\right)=(-i)^n J_n^{ic}(t), \quad  \overline{I}_{n}^{c}\left( t\right)=(-i)^n\overline{J}_{n}^{ic}\left( t\right),  \quad \text{for} \quad c\in \mathbb{C}\setminus\{\alpha:\, \alpha\in \mathbb{R},\, |\alpha|\geq 1\}.
$$

When $n>t$, it is well known that $I_n^c(t)=J_n^c(t)=0$. Therefore, for $t\in\mathbb{Z}_{<0}$, for a suitable range of $c$ we have that $n>-t$ implies $\overline{J}_{n}^{c}\left( t\right)= \overline{I}_{n}^{c}\left( t\right)=0$.

The following proposition shows that for $t\in\mathbb{N}_0$, the function $J_{n}^{c}\left( t\right) $ can be viewed as a polynomial in variable $c$, while for $t\in\mathbb{Z}_{<0}$  functions $\overline{J}_{n}^{c}\left( t\right)$  and $\overline{I}_{n}^{c}\left( t\right)$ can be viewed as a polynomial in $c$.

\begin{proposition}
\label{propJ}\hfill
\begin{enumerate}[label=\normalfont(\roman*)]
\item Let $t,n\in \mathbb{N}_{0}$ such that $n\leq t$. Set $\ell
=\left\lfloor \left( t-n\right) /2\right\rfloor $ . Then for any $c\in%
\mathbb{C}\setminus\{0\}$, we have that
\begin{equation}\label{eq. J as polynom}
J_{n}^{c}\left( t\right) =\sum\limits_{k=0}^{\ell }\frac{\left( -1\right)
^{k}t!}{k!\left( t-2k-n\right) !\left( n+k\right) !}\left( \frac{c}{2}%
\right) ^{2k+n}.
\end{equation}

\item Let $t\in\mathbb{Z}_{<0}$, $c\in\mathbb{C}\setminus\{0\}$ and $n\in \mathbb{N}_{0}$ such that $n\leq -t$.
Set $\ell =\left\lfloor \left( -t-n\right) /2\right\rfloor$. Then, we have that
\begin{equation*}
\overline{J}_{n}^{c}\left( t\right) =\sum\limits_{k=0}^{\ell }\frac{\left(
-1\right) ^{k+n}\left( -t\right) !}{k!\left( -t-2k-n\right) !\left(
n+k\right) !}\left( \frac{c}{2}\right) ^{2k+n},
\end{equation*}
and
\begin{equation*}
\overline{I}_{n}^{c}\left( t\right) =(-1)^n \sum\limits_{k=0}^{\ell }\frac{\left( -t\right) !}{k!\left( -t-2k-n\right) !\left(
n+k\right) !}\left( \frac{c}{2}\right) ^{2k+n}.
\end{equation*}
\end{enumerate}
\end{proposition}
\begin{proof}
According to \cite[Proposition 3.2.]{CHJSV}, for any $c\in \mathbb{C} $, we
have
\begin{equation*}
I_{n}^{c}\left( t\right) =\sum\limits_{k=0}^{\ell }\frac{t!}{k!\left(
t-2k-n\right) !\left( n+k\right) !}\left( \frac{c}{2}\right) ^{2k+n}.
\end{equation*}
Combining it with $I_{n}^{c}\left( t\right) =\left( -i\right)
^{n}J_{n}^{ic}\left( t\right) $, we easily deduce \eqref{eq. J as polynom}, which proves part (i). Part (ii) stems from the identities $\overline{J}
_{n}^{c}\left( t\right) =\left( -1\right) ^{n}J_{n}^{c}\left( -t\right) $  and $\overline{I}
_{n}^{c}\left( t\right) =\left( -1\right) ^{n}I_{n}^{c}\left( -t\right) $.
\end{proof}

Transformation formulas for $\partial_t J_n^c$, $\partial_t I_n^c$ and $\overline{\partial}_t\overline{I}_n^c$ have been deduced in \cite{Slav18}, \cite{CHJSV} and \cite{KS22}, respectively (see also \cite{CGW23} for recurrence formulas satisfied by matrix analogues of $J_n^1$). In the following lemma we prove transformation formulas for $\overline{\partial}_t\overline{J}_n^c$.
\begin{lemma}
\label{lemma2} For $c\in \mathbb{C}\setminus\{i\alpha:\, \alpha\in \mathbb{R},\, |\alpha|\geq 1\}$, the discrete $J$-Bessel function $%
\overline{J}_{n}^{c}$ has the following properties:

\begin{enumerate}[label=\normalfont(\roman*)]

\item $\overline{J}_{0}^{c}\left( 0\right) =1$.

\item $\overline{\partial }_{t}\overline{J}_{0}^{c}\left( t\right) =-c%
\overline{J}_{1}^{c}\left( t\right) $ for all $t\geq 1$.

\item $t\overline{\partial }_{t}\overline{J}_{n}^{c}\left( t+1\right) =n%
\overline{J}_{n}^{c}\left( t\right) -ct\overline{J}_{n+1}^{c}\left(
t+1\right) $ for any $n\geq 0$ and $t\geq 0$.

\item $t\overline{\partial }_{t}\overline{J}_{n}^{c}\left( t+1\right) =-n%
\overline{J}_{n}^{c}\left( t\right) +ct\overline{J}_{n-1}^{c}\left(
t+1\right) $ for any $n\geq 1$ and $t\geq 0$.

\item $\overline{\partial }_{t}\overline{J}_{n}^{c}\left( t\right) =\frac{c}{%
2}\left( \overline{J}_{n-1}^{c}\left( t\right) -\overline{J}_{n+1}^{c}\left(
t\right) \right) $ for any $n\geq 1$ and $t\geq 1$.
\end{enumerate}

\begin{proof}
The first statement follows from the definition of $\overline{J}%
_{n}^{c}\left( t\right) $.

Formula (ii) can be written as%
\begin{equation*}
\overline{J}_{0}^{c}\left( t\right) -\overline{J}_{0}^{c}\left( t-1\right)
=-c\overline{J}_{1}^{c}\left( t\right) .
\end{equation*}

Using recursive formula (ii) from Lemma \ref{GHFrecurs} with $\alpha =\frac{%
t-1}{2},\beta =\frac{t}{2},\gamma =1$ and $z=-c^{2}$, and the symmetry of $F$
in the first two arguments, we obtain
\begin{equation*}
F\left( \frac{t-1}{2},\frac{t}{2};1;-c^{2}\right) -F\left( \frac{t}{2},\frac{%
t+1}{2};1;-c^{2}\right) =\frac{tc^{2}}{2}F\left( \frac{t+1}{2},\frac{t+2}{2}%
;2;-c^{2}\right) .
\end{equation*}%
Therefore,
\begin{multline*}
F\left( \frac{t}{2},\frac{t}{2}+\frac{1}{2};1;-c^{2}\right) -F\left( \frac{%
t-1}{2},\frac{t-1}{2}+\frac{1}{2};1;-c^{2}\right)
\\ =-c\frac{tc}{2}F\left( \frac{t+1}{2},\frac{t+1}{2}+\frac{1}{2}%
;2;-c^{2}\right) ,
\end{multline*}
which proves (ii) for all $t\geq 1$.

The identity in part (iii) can be written as%
\begin{equation} \label{eq. fla iii}
t\overline{J}_{n}^{c}\left( t+1\right) -\left( n+t\right) \overline{J}%
_{n}^{c}\left( t\right) +ct\overline{J}_{n+1}^{c}\left( t+1\right) =0.
\end{equation}
Using recursive relation (i) from Lemma \ref{GHFrecurs} with $\alpha =\frac{%
n+t+1}{2},\beta =\frac{n+t}{2},\gamma =n+1$ and $z=-c^{2} $, and the
symmetry of $_{2}F_{1}$ in the first two arguments, we obtain
\begin{multline*}
\left( n+1\right) F\left( \frac{n+t}{2} ,\frac{n+t}{2}+\frac{1}{2}
;n+1;-c^2\right) \\-\left( n+1\right) F\left( \frac{n+t+1}{2} ,\frac{n+t+1}{2}+\frac{1}{2}
;n+1;-c^2\right) \\
-\frac{n+t+1}{2}c^{2} F\left( \frac{n+t+2}{2} ,\frac{n+t+2}{2}+\frac{1}{2}
;n+2;-c^2\right)=0.
\end{multline*}
Multiplying the above display by $\frac{c^{n}}{2^{n}\left( n+1\right) !}\left( t\right) _{n+1}$%
, and using the recurrent relations $\left( t\right) _{n+1}=t\left( t+1\right) _{n}$, $\left(
t\right) _{n+1}=\left( n+t\right) \left( t\right) _{n}$, and
$\left( n+t+1\right) \left( t\right) _{n+1}=t\left( t+1\right) _{n+1}$, we deduce \eqref{eq. fla iii}.
\vskip .06in
The proof of part (iv) is analogous to the proof of (iii); it follows from a simple manipulation of the relation (iii) given in Lemma \ref{GHFrecurs} with $\alpha =\frac{n+t}{2},\beta =\frac{n+t}{2}+\frac{1}{2}%
,\gamma =n$ and $z=-c^{2}$.


\vskip .06in
Formula (v) is deduced by adding formulas (iii) and (iv), dividing the result by $2t$, and replacing $t$ by $t-1$.
\end{proof}
\end{lemma}


We will now explore asymptotic behaviour of $\overline{J}_{n}^{c}\left(
t\right) $ and $\overline{I}_{n}^{c}\left( t\right) $ when $c$ is real-valued, and as $n\rightarrow
\infty $.

\begin{proposition}
\label{asympJn}\hfill

\begin{enumerate}[label=\normalfont(\roman*)]

\item For any real nonzero parameter $c$ and a fixed $t\in \mathbb{N}$, we
have that%
\begin{equation}
\overline{J}_{n}^{c}\left( t\right) \sim \left( \mathrm{sgn}(c)\right) ^{n}%
\frac{n^{t-1}}{\left( \frac{1+\sqrt{1+c^{2}}}{\left\vert c\right\vert }%
\right) ^{n}}\frac{(1+c^2)^{-t/2}}{\Gamma \left( t\right) },
\text{ as }n\rightarrow \infty .  \label{nasympJ}
\end{equation}

\item For any real nonzero parameter $c$ such that $\left\vert c\right\vert <1$ and a
fixed $t\in \mathbb{N}$, we have that%
\begin{equation}
\overline{I}_{n}^{c}\left( t\right) \sim \left( \mathrm{sgn}(c)\right) ^{n}%
\frac{n^{t-1}}{\left( \frac{1+\sqrt{1-c^{2}}}{\left\vert c\right\vert }%
\right) ^{n}}\frac{(1-c^2)^{-t/2}}{\Gamma \left( t\right) },
\text{ as }n\rightarrow \infty .  \label{nasympI}
\end{equation}
\end{enumerate}

\begin{proof}
i) We will use the asymptotic formula \eqref{eq: Watson fla 1}
with $\lambda =\frac{n}{2}$, $\alpha =\beta =\frac{t}{2}$, $\gamma =\frac{1}{%
2}$, $z=1+\frac{2}{c^{2}}>1$ and
$z+\sqrt{z^{2}-1}=e^{\zeta }=\frac{\left( 1+\sqrt{1+c^{2}}\right) ^{2}}{
c^{2}}.$
This, together with the asymptotic formula \cite[formula (6.1.39)]{AS} for the gamma function
\begin{equation}
\Gamma \left( aw+b\right) \sim \sqrt{2\pi }e^{-aw}\left( aw\right) ^{aw+b-%
\frac{1}{2}},\text{ \ }\left( \left\vert \arg w\right\vert <\pi ,a>0\right) ,\quad |w|\to \infty
\label{gammaasymp}
\end{equation}
and the fact that the series $\sum\limits_{s=1}^{\infty }c_{s}^{\prime }%
\frac{\Gamma \left( s+\frac{1}{2}\right) }{n^{s}}$ is convergent yields the following asymptotics

\begin{eqnarray*}
\overline{J}_{n}^{c}\left( t\right)
&\sim &\frac{\left( \mathrm{sgn}(c)\right) ^{n}2^{t}}{2^{n}\Gamma \left(
t\right) {\left\vert c\right\vert }^{n+t}}\frac{\sqrt{2\pi }e^{-n}n^{n+t-%
\frac{1}{2}}}{2\pi e^{-n}\left( \frac{n}{2}\right) ^{n}}\left( \frac{c^{2}}{%
\left( 1+\sqrt{1+c^{2}}\right) ^{2}}\right) ^{\frac{n+t}{2}}\\
&& \times \left( 1+\frac{%
c^{2}}{\left( 1+\sqrt{1+c^{2}}\right) ^{2}}\right) ^{-t}\cdot \frac{\sqrt{2\pi }}{%
\sqrt{n}}\left( 1+O\left( \frac{1}{n}\right) \right)  \\
&=&\left( \mathrm{sgn}(c)\right) ^{n}\frac{n^{t-1}}{\left( \frac{1+\sqrt{%
1+c^{2}}}{\left\vert c\right\vert }\right) ^{n}}\frac{1}{\Gamma \left(
t\right) }\left( \frac{2\left( 1+\sqrt{1+c^{2}}\right) }{\left( 1+\sqrt{%
1+c^{2}}\right) ^{2}+c^{2}}\right) ^{t}\left( 1+O\left( \frac{1}{n}\right)
\right),
\end{eqnarray*}
as $n\rightarrow \infty$. This proves \eqref{nasympJ}.

ii) Assume $c\in \mathbb{R}$ such that $0<\left\vert c\right\vert <1$.
We apply equation \eqref{eq. hypergom transf} with $\alpha =\frac{n+t}{2}$, $\beta =\frac{n+t}{2}+\frac{1}{2}$, $\gamma =n+1$ and $z=c^{2}$ to deduce
\begin{equation*}
F\left( \frac{n+t}{2},\frac{n+t}{2}+\frac{1}{2};n+1;c^{2}\right) =\left(
1-c^{2}\right) ^{-\frac{n+t}{2}}F\left( \frac{n+t}{2},\frac{n-t}{2}+\frac{1}{%
2};n+1;\frac{c^{2}}{c^{2}-1}\right),
\end{equation*}
hence
\begin{equation}\label{eq. I bar repr}
\overline{I}_{n}^{c}\left( t\right) =\frac{\left( c/2\right) ^{n}\Gamma(t+n)}{ n! \Gamma(t)}\left(
1-c^{2}\right) ^{-\frac{n+t}{2}}F\left( \frac{n+t}{2},\frac{n-t}{2}+\frac{1}{%
2};n+1;\frac{c^{2}}{c^{2}-1}\right).
\end{equation}
We can write $\frac{c^{2}}{c^{2}-1}=\frac{2}{1-z}$ for $z=\frac{2-c^{2}}{%
c^{2}}>1$. Hence,
$$z+\sqrt{z^{2}-1}=e^{\zeta }=\frac{\left( 1+\sqrt{%
1-c^{2}}\right) ^{2}}{c^{2}},$$
for $\zeta >0$.

Now, we proceed in the same ways as above, i.e., apply the asymptotic formula  \eqref{eq: Watson fla 1} with $\lambda =\frac{n}{2}$, $\alpha =\beta =\frac{t}{2}$, $\gamma =t + \frac{1}{2}$ and use the asymptotic behaviour \eqref{gammaasymp} of the gamma function to derive  \eqref{nasympI} for $0<\left\vert
c\right\vert <1$.
\end{proof}
\end{proposition}

Since $\frac{1+\sqrt{1+c^{2}}}{\left\vert c\right\vert
}>1$ for all real nonzero $c$ and $\frac{1+\sqrt{1-c^{2}}}{\left\vert c\right\vert} >1$ for all real nonzero $c$ such that $\left\vert c\right\vert <1$, it is obvious that the right-hand sides of equations \eqref{nasympJ} and \eqref{nasympI}, for a fixed positive integer $t$,  decay exponentially as $n\to \infty$. This proves the following corollary.
\begin{corollary}
\hfill
\begin{enumerate}[label=\normalfont(\roman*)]
 \item For  any real nonzero parameter $c$ and a fixed $t\in \mathbb{N}$, we
have that $\overline{J}_{n}^{c}\left( t\right) \rightarrow 0$ as $
n\rightarrow \infty $.
  \item For any real nonzero parameter $c$, $\left\vert c\right\vert <1$ and a
fixed $t\in \mathbb{N}$, we have that $\overline{I}_{n}^{c}\left( t\right) \rightarrow 0$ as $%
n\rightarrow \infty $.
\end{enumerate}
  \end{corollary}

\vspace{1.5cc}
\section{Proofs of main results}\label{main_section}

In this section, we will prove our main results. First, we will show that our
backward discrete Bessel functions satisfy the corresponding backward difference
equations.

\setcounter{subsection}{0}
\subsection{Proof of Theorem \protect\ref{thm: difference eq.}}

Consider the backward difference equation \eqref{backBesselEq1}. By
expanding the differences, one can obtain the following equivalent form of the equation
\eqref{backBesselEq1} with the $+$ sign
\begin{equation}
t\left( t+1\right) (1+c^{2})y_{n}\left( t+2\right) -t\left( 2t+1\right)
y_{n}\left( t+1\right) -\left( n^{2}-t^{2}\right) y_{n}\left( t\right) =0.
\label{backBesselEq2}
\end{equation}

Since $\overline{I}_{n}^{c}(t) = \left(-i\right)^n \overline{J}_{n}^{ic} (t)$%
, it is enough to prove that $\overline{J}_{n}^{c}\left( t\right) $
satisfies the difference equation \eqref{backBesselEq2}.

Using relation (iv) from Lemma \ref{GHFrecurs} with $\alpha =\frac{n+t}{2}%
+1,\beta =\frac{n+t}{2}+\frac{1}{2},\gamma =n+1$ and $z=-c^{2}$, we get%
\begin{multline*}
-\left( t+\frac{1}{2}\right) F\left( \frac{n+t}{2}+1,\frac{n+t}{2}+\frac{1}{2%
};n+1;-c^{2}\right) \\
-\frac{n-t}{2}F\left( \frac{n+t}{2},\frac{n+t}{2}+\frac{1}{2}%
;n+1;-c^{2}\right) \\
+\frac{n+t+1}{2}(1+c^{2})F\left( \frac{n+t+2}{2},\frac{n+t+2}{2}+\frac{1}{2}%
;n+1;-c^{2}\right) =0.
\end{multline*}

Multiplying by $\frac{2c^{n}}{2^{n}n!}\left( t\right) _{n+1}$ and using the
definition of $\overline{J}_{n}^{c}\left( t\right) $, we obtain%
\begin{equation*}
-t\left( 2t+1\right) \overline{J}_{n}^{c}\left( t+1\right) -\left(
n^{2}-t^{2}\right) \overline{J}_{n}^{c}\left( t\right) +t\left( t+1\right)
(1+c^{2})\overline{J}_{n}^{c}\left( t+2\right) =0,
\end{equation*}%
which is \eqref{backBesselEq2}.

\subsection{Proof of Theorem \protect\ref{th_asymp}}

First, we will prove the asymptotic formula \eqref{asymp1}. Applying the identity \eqref{eq. hypergeom = Leg} with $a=(n-t)/2$, $c=n+1$, and $x=-c^2<0$ we get
\begin{equation}\label{eq. F as P}
F\left( \frac{n-t}{2},\frac{n-t}{2}+\frac{1}{2};n+1;-c^{2}\right) =2^{n}n!\frac{\left( 1+c^{2}\right) ^{\frac{t}{2}}}{\left\vert c\right\vert ^{n}}%
P_{-t-1}^{-n}\left( \left( 1+c^{2}\right) ^{-\frac{1}{2}}\right).
\end{equation}%
We put $\cos \theta =\left( 1+c^{2}\right) ^{-\frac{1}{2}}\in \left(
0,1\right) $ and use \eqref{eq. Leg at negatives} with $\nu=t$ and $m=n$ to deduce
\begin{equation}\label{P neg as P}
P_{-t-1}^{-n}\left( \cos \theta \right) =P_{t}^{-n}\left( \cos \theta
\right) =\frac{\Gamma \left( t-n+1\right) }{\Gamma \left( t+n+1\right) }%
P_{t}^{n}\left( \cos \theta \right).
\end{equation}
Combining \eqref{eq. F as P} and \eqref{P neg as P} with the asymptotic formula \eqref{eq. Leg asymptotics} (in which we take $\nu=t$ and $\mu=n$) we arrive at
$$
F\left( \frac{n-t}{2},\frac{n-t}{2}+\frac{1}{2};n+1;-c^{2}\right)\sim \frac{2^{n+1}n!\left( 1+c^{2}\right) ^{\frac{t}{2}}}{\left\vert c\right\vert ^{n}\sqrt{\pi}}\frac{\Gamma
\left( t-n+1\right) }{\Gamma \left( t+\frac{3}{2}\right) }\frac{\cos \left(
\left( t+\frac{1}{2}\right) \theta -\frac{\pi }{4}+\frac{n\pi }{2}\right) }{%
\sqrt{2\sin \theta }},
$$
as $t\to\infty$.

Since $\sin \theta =\frac{\left\vert c\right\vert }{\sqrt{1+c^{2}}}$, we deduce the following asymptotic formula for $J_{n}^{c}\left( t\right) $ as $t\to\infty$:
\begin{multline}\label{eq. J asympt}
J_{n}^{c}\left( t\right) =\frac{\left( -c/2\right) ^{n}\left( -t\right)
_{n}}{n!}F\left( \frac{n-t}{2},\frac{n-t}{2}+\frac{1}{2};n+1;-c^{2}\right) \\
\sim \frac{\sqrt{2}}{\sqrt{\pi }}\left( -1\right) ^{n}\left( \mathrm{sgn}%
c\right) ^{n}\frac{\left( -t\right) _{n}\Gamma \left( t-n+1\right) }{\Gamma
\left( t+\frac{3}{2}\right) }\frac{\left( 1+c^{2}\right) ^{\frac{t}{2}+\frac{%
1}{4}}}{\sqrt{\left\vert c\right\vert }}\cos \left( \left( t+\frac{1}{2}%
\right) \theta -\frac{\pi }{4}+\frac{n\pi }{2}\right).
\end{multline}

Using the asymptotic formula \eqref{gammaasymp} for the gamma function,  we get
$$
(-1)^n \frac{\left( -t\right) _{n}\Gamma \left( t-n+1\right) }{\Gamma
\left( t+\frac{3}{2}\right) }=\frac{\Gamma(t+n) \Gamma(t-n+1)}{\Gamma(t)\Gamma(t+3/2)}\sim \frac{1}{\sqrt{t}}, \text{  as  } t\to \infty.
$$
Inserting this into \eqref{eq. J asympt} we arrive at \eqref{asymp1}.\medskip



Now, we prove \eqref{asymp2}. Applying the identity \eqref{eq. hypergeom = Leg} with $a=(n+t)/2$, $c=n+1$, and $x=-c^2<0$ together with the identity \eqref{eq. Leg at negatives} we get
\begin{equation*}
F\left( \frac{n+t}{2} ,\frac{n+t}{2}+\frac{1}{2} ;n+1 ;-c^2 \right)
= 2^{n}n!\frac{\left( 1+c^{2}\right) ^{-\frac{t}{2}}}{\left\vert
c\right\vert ^{n}}\frac{\Gamma \left( t-n\right) }{\Gamma \left( t+n\right) }%
P_{t-1}^{n}\left( \left( 1+c^{2}\right) ^{-\frac{1}{2}}\right) .
\end{equation*}

The asymptotic formula \eqref{eq. Leg asymptotics} with $\nu=t-1$ and $\mu=n$ yields
\begin{eqnarray*}
\overline{J}_{n}^{c}\left( t\right) &=&\frac{\left( c/2\right) ^{n}\left( t\right)
_{n}}{n!}F\left( \frac{n+t}{2},\frac{n+t}{2}+\frac{1}{2};n+1;-c^{2}\right) \\
&\sim &\frac{\sqrt{2}}{\sqrt{\pi }}\left( \mathrm{sgn}%
c\right) ^{n}\frac{\Gamma(t)}{\Gamma
\left( t+\frac{1}{2}\right) }\frac{\left( 1+c^{2}\right) ^{-\frac{t}{2}+\frac{1}{4}}}{\sqrt{%
\left\vert c\right\vert }}\cos \left( \left( t-\frac{1}{2}\right) \theta -%
\frac{\pi }{4}+\frac{n\pi }{2}\right),
\end{eqnarray*}
as $t\to\infty$. Applying the asymptotic formula \eqref{gammaasymp} to the quotient $\frac{\Gamma(t)}{\Gamma
\left( t+\frac{1}{2}\right) }$ in the above display immediately yields \eqref{asymp2}.

\medskip

It remains to prove  \eqref{asymp3}. Assume $c\in\mathbb{R}$, $|c|<1$. Our starting point is equation \eqref{eq. I bar repr}, in which we would like to deduce the asymptotics for the hypergeometric function as $t\to\infty$. In order to do so, we write $\frac{c^{2}}{c^{2}-1}=\frac{1}{2}\left( 1-z\right) $ for $z=\frac{1+c^{2}}{1-c^{2}}>1$. Hence
$$z+\sqrt{z^{2}-1}=e^{\zeta }=\frac{%
1+\left\vert c\right\vert }{1-\left\vert c\right\vert },
$$
for $\zeta >0$, which justifies application of the second asymptotic Watson's formula \eqref{eq: Watson fla 2} with $\alpha =\frac{n}{2}$, $\beta =\frac{n+1}{2}$, $\gamma =n+1$, and $\lambda=\frac{t}{2}$. In this case $e^{-\zeta}<1$, hence the second term in the sum in \eqref{eq: Watson fla 2} decays exponentially as $\lambda=t/2\to\infty$, meaning that the first sum gives the lead term in asymptotics. We get
\begin{multline*} \label{eq. hyperg asympt}
F\left( \frac{n+t}{2},\frac{n-t}{2}+\frac{1}{2};n+1;\frac{c^{2}}{c^{2}-1}\right)\\ \sim  \frac{n! \Gamma\left( \frac{t-n+1}{2}\right)}{\pi \Gamma\left( \frac{t+n+1}{2}\right)}2^{n-1/2}\left(\frac{1+|c|}{2|c|}\right)^{n+1/2}\left(\frac{1+|c|}{1-|c|}\right)^{(t-n-1)/2}
\left(\frac{\sqrt{\pi}}{\sqrt{t/2}}+O\left(\frac{1}{t}\right)\right),
\end{multline*}
as $t\to\infty$. Inserting this into the formula \eqref{eq. I bar repr} we deduce
\begin{equation*}
\overline{I}_{n}^{c}\left( t\right) \sim \left( \mathrm{sgn}(c)\right) ^{n}\frac{2^{-n}\Gamma \left(
t+n\right) \Gamma \left( \frac{t-n+1}{2}\right) }{\Gamma \left( t\right)
\Gamma \left( \frac{t+n+1}{2}\right) \sqrt{2\pi t\left\vert c\right\vert }}%
\left( 1-\left\vert c\right\vert \right) ^{-t+\frac{1}{2}},\text{as }%
t\rightarrow \infty .
\end{equation*}

Formula \eqref{gammaasymp} yields that $\frac{\Gamma \left(
t+n\right) \Gamma \left( \frac{t-n+1}{2}\right) }{\Gamma \left( t\right)
\Gamma \left( \frac{t+n+1}{2}\right) }\sim 2^{n}$ as $t\rightarrow \infty $,
which completes the proof of \eqref{asymp3}.

\subsection{Proof of Theorem \protect\ref{thm: Laplace tr}}

Let $c\in \mathbb{C}\setminus \left\{ 0\right\} $ and $n\in \mathbb{N}_{0}$. We begin by recalling the result of \cite{CHJSV} which states that for any $n\in\mathbb{N}$ the generating function
$$
g_n^c(z) := \sum\limits_{t=0}^{\infty }I_{n}^{c}\left( t\right) z^{t}
$$
of the sequence $\left\{ I_{n}^{c}\left( t\right) \right\} _{n\in
\mathbb{N}}$ is holomorphic in the disc $|z|<\frac{1}{1+|c|}$ and possesses meromorphic continuation to the whole $z$-plane given by
$$
g_n^c(z)= \frac{1}{\sqrt{(1-z)^{2}-c^{2}z^{2}}}\left(
\frac{(1-z) -\sqrt{\left( 1-z\right) ^{2}-c^{2}z^{2}}}{cz}\right) ^{n}.
$$
According to the asymptotic relation \eqref{eq. I asympt as t to infty}, the Laplace transform  $\mathcal{L}_{\partial _{t}}$ of the sequence  $\left\{ I_{n}^{c}\left( t\right) \right\} _{n\in\mathbb{N}}$ is a holomorphic function in the region $|1+z|> 1+|c|$ and
$$
\mathcal{L}_{\partial _{t}}\{I_{n}^{c}\}(z)=\sum_{t=0}^{\infty }\frac{%
I_{n}^{c}\left( t\right) }{(1+z)^{t+1}}=\frac{1}{1+z}g_{n}^{c}\left( \frac{1}{1+z}\right)=\frac{c^{-n}\left(z- \sqrt{
z^{2}-c^{2}}\right) ^{n}}{\sqrt{z^{2}-c^{2}}}.
$$
The right-hand side provides the meromorphic continuation of $\mathcal{L}_{\partial _{t}}\{I_{n}^{c}\}(z)$ to all complex values of $z$ such that $z\neq \pm c$.

The obvious inequality $\left\vert J_{n}^{c}\left( t\right) \right\vert \leq
J_{n}^{\left\vert c\right\vert }\left( t\right) $ for all integers $n,t\geq
0 $, combined with the asymptotic formula \eqref{asymp1} implies that the radius of convergence of the
power series
\begin{equation}
\sum\limits_{t=0}^{\infty }J_{n}^{c}\left( t\right) z^{t}  \label{gen1}
\end{equation}%
equals $\frac{1}{\sqrt{1+\left\vert c\right\vert ^{2}}}$. Therefore, the
series \eqref{gen1} defines a holomorphic function $f_{n}^{c}\left(
z\right) $ in the disc $|z|<\frac{1}{\sqrt{1+\left\vert c\right\vert ^{2}}}$ which is the generating function  of the sequence $\left\{ J_{n}^{c}\left( t\right) \right\} _{n\in
\mathbb{N}}$.

The identity $I_{n}^{c}\left(
t\right) =\left( -i\right) ^{n}J_{n}^{ic}\left( t\right) $ which is valid for all nonzero complex $c$, when $n,t \geq 0$ yields $J_{n}^{c}\left( t\right) =i^{n}I_{n}^{-ic}\left( t\right) $. Therefore, $f_n^c(z)= i^n g_n^{-ic}(z)$, for all $z$ in the disc $|z|<\frac{1}{1+|c|} \leq \frac{1}{\sqrt{1+\left\vert c\right\vert ^{2}}}$.
A simple computation that amounts to algebraic manipulations and choosing the principal branch of the square root yields that
\begin{equation}
f_{n}^{c}\left( z\right) =\frac{1}{\sqrt{(z-1)^{2}+c^{2}z^{2}}}\left( \frac{%
z-1}{cz}+\sqrt{\frac{\left( z-1\right) ^{2}}{c^{2}z^{2}}+1}\right) ^{n},
\label{genfunc}
\end{equation}%
for any $c\in \mathbb{C}\setminus \left\{ 0\right\} $, $n\in
\mathbb{N}_{0}$ and $\left\vert z\right\vert <\frac{1}{\sqrt{1+\left\vert
c\right\vert ^{2}}}$. The right-hand side of \eqref{genfunc} provides the meromorphic continuation of $f_{n}^{c}\left( z\right)$ to the whole $z$-plane.

The asymptotic formula \eqref{asymp1} ensures that the Laplace transform $\mathcal{L}_{\partial _{t}}$ of the sequence  $\left\{ J_{n}^{c}\left( t\right) \right\} _{n\in\mathbb{N}}$ as given in \eqref{LaplaceFW} is well defined and holomorphic function in the region $|1+z|>\sqrt{1+|c|^2}$ and moreover
\begin{equation*}
\mathcal{L}_{\partial _{t}}\{J_{n}^{c}\}(z)=\sum_{t=0}^{\infty }\frac{%
J_{n}^{c}\left( t\right) }{(1+z)^{t+1}}=\frac{1}{1+z}f_{n}^{c}\left( \frac{1%
}{1+z}\right) =\frac{c^{-n}\left( \sqrt{z^{2}+c^{2}}-z\right) ^{n}}{\sqrt{%
z^{2}+c^{2}}}.
\end{equation*}%
The right-hand side provides the holomorphic continuation of $\mathcal{L}_{\partial _{t}}\{J_{n}^{c}\}(z)$ to all complex values of $z$ such that $z\neq \pm ic$.

This proves the claim of the theorem for the Laplace transform $\mathcal{L}_{\partial _{t}}$ associated to the forward difference operator. The Laplace transform for the backward difference operator is evaluated analogously.

Namely, assume $c\in\mathbb{C} \setminus\{\alpha:\, \alpha\in\mathbb{R}, |\alpha|\geq 1\}$. In \textrm{\cite[formula (3.28)]{KS22}} Kan and Shiraishi computed the generating function for the sequence $\left\{ \overline{I}_{n}^{c}\left( t\right) \right\} _{n\in \mathbb{N}}$; its meromorphic continuation to the whole $z$-plane is given for $n\geq 0$ by
$$
\overline{g}_n^c(z)= \frac{z}{\sqrt{(1-z)^{2}-c^{2}}}\left(
\frac{(1-z)}{c} -\sqrt{\frac{\left( 1-z\right) ^{2}-c^{2}}{c^2}}\right) ^{n}.
$$
The definition \eqref{LaplaceBW} of the Laplace transform associated with the backward difference operator, combined with the asymptotic formula \eqref{asymp3}  yields that for complex values $c$ in the unit disc and for $z\in\mathbb{C}$ such that $|1-z| < 1-|c|$ we have
\begin{equation*}
\mathcal{L}_{\overline{\partial }_{t}}\{\overline{I}_{n}^{c}\}(z)=%
\sum_{t=0}^{\infty }\overline{I}_{n}^{c}\left( t\right) (1-z)^{t-1}=\frac{1}{%
1-z}\overline{g}_{n}^{c}(1-z)=\frac{c^{-n}\left(z- \sqrt{z^{2}-c^{2}}\right)
^{n}}{\sqrt{z^{2}-c^{2}}}.
\end{equation*}%
The right-hand side of the above equation provides the meromorphic continuation of $\mathcal{L}_{\overline{\partial }_{t}}\{\overline{I}_{n}^{c}\}(z)$ to all complex, nonzero $c$ and all $z\in\mathbb{C}$ with $z\neq \pm c$.

Computation of $\mathcal{L}_{\overline{\partial }_{t}}\{\overline{J}_{n}^{c}\}(z)$ is analogous, so we provide only the two key steps. The first step is computation of the generating function $\overline{f}_{n}^{c}\left( z\right)$ of the sequence $\left\{ \overline{J}_{n}^{c}\left( t\right) \right\} _{n\in \mathbb{N}}$ which follows by using the identity $\overline{J}_{n}^{c}\left( t\right) =i^{n}\overline{I}%
_{n}^{-ic}\left( t\right) $ to relate $\overline{f}_{n}^{c}\left( z\right)$  to $\overline{g}_{n}^{ic}\left( z\right)$ and deduce that
\begin{equation*}
\overline{f}_{n}^{c}\left( z\right) =\frac{z}{\sqrt{(z-1)^{2}+c^{2}}}\left(
\frac{z-1}{c}+\sqrt{\frac{(z-1)^{2}}{c^{2}}+1}\right) ^{n}.
\label{modgenfunc}
\end{equation*}
The second step is the evaluation of the Laplace transform, which follows from the observation that
$$
\mathcal{L}_{\overline{\partial }_{t}}\{\overline{J}_{n}^{c}\}(z)=\frac{1}{
1-z}\overline{f}_{n}^{c}(1-z)=\frac{c^{-n}\left( \sqrt{z^{2}+c^{2}}-z\right)
^{n}}{\sqrt{z^{2}+c^{2}}}.
$$

\vspace{1.5cc}
\begin{section}
{Backward discrete wave equation}
\end{section}\label{wave_section}

In this section, we will study the backward discrete wave equation
\begin{equation}
\overline{\partial }_{t}^{2}u\left( n;t\right) =c^{2}\left( u\left(
n+1;t\right) -2u\left( n;t\right) +u\left( n-1;t\right) \right) \text{, \ }%
n\in \mathbb{Z}\text{, \ }t\in \mathbb{N}_{0}\text{,}  \label{backWaveEq}
\end{equation}
which is the backward analogue of \eqref{discwaveq1} and find its fundamental and general solutions under natural initial conditions.
Then, we will study the asymptotic behaviour of the first fundamental solutions of both
forward and backward discrete wave equations when the time variable tends to infinity.

\setcounter{subsection}{0}
\subsection{Fundamental solutions to the backward discrete wave equation}

The first fundamental solution to the backward discrete wave equation is described in the following theorem.

\begin{theorem}\label{thm. fund sol}
Let $c>0$. The solution of the backward wave equation \eqref{backWaveEq}
with initial conditions
\begin{equation}
u\left( n;0\right) =\left\{
\begin{array}{ll}
1 & \text{if }n=0, \\
0 & \text{if }n\neq 0,%
\end{array}%
\right. ,\quad \overline{\partial }_{t}u\left( n;0\right) =0,\text{ \ }n\in
\mathbb{Z},  \label{icWaveEq}
\end{equation}
is given by
\begin{equation}
u\left( n;t\right) =\overline{J}_{2\left\vert n\right\vert }^{2c}\left(
t\right) \text{, \ }n\in \mathbb{Z}\text{, \ }t\in \mathbb{N}_{0}\text{.\ }
\label{solWaveEq}
\end{equation}

\begin{proof}
Let us define $u\left( n;t\right) $ by \eqref{solWaveEq}. It is easy to
verify that $\overline{J}_{2\left\vert n\right\vert }^{2c}\left( t\right)$
satisfies initial conditions \eqref{icWaveEq}. We will now check if $%
\overline{J}_{2\left\vert n\right\vert }^{2c}\left( t\right) $ satisfies %
\eqref{backWaveEq}. For $n\geq 1$, using Lemma \ref{lemma2} (v) we obtain%
\begin{eqnarray*}
\overline{\partial }_{t}u\left( n;t\right) &=&\overline{\partial }_{t}%
\overline{J}_{2n}^{2c}\left( t\right) =c\left( \overline{J}%
_{2n-1}^{2c}\left( t\right) -\overline{J}_{2n+1}^{2c}\left( t\right) \right)
, \\
\overline{\partial }_{t}^{2}u\left( n;t\right) &=&\overline{\partial }%
_{t}^{2}\overline{J}_{2n}^{2c}\left( t\right) =c^{2}\left( \overline{J}%
_{2n-2}^{2c}\left( t\right) -2\overline{J}_{2n}^{2c}\left( t\right) +%
\overline{J}_{2n+2}^{2c}\left( t\right) \right) \\
&=&c^{2}\left( u\left( n-1;t\right) -2u\left( n;t\right) +u\left(
n+1;t\right) \right) .
\end{eqnarray*}

Similarly, if $n\leq -1$, we obtain%
\begin{eqnarray*}
\overline{\partial }_{t}u\left( n;t\right) &=&\overline{\partial }_{t}%
\overline{J}_{-2n}^{2c}\left( t\right) =c\left( \overline{J}%
_{-2n-1}^{2c}\left( t\right) -\overline{J}_{-2n+1}^{2c}\left( t\right)
\right) , \\
\overline{\partial }_{t}^{2}u\left( n;t\right) &=&\overline{\partial }%
_{t}^{2}\overline{J}_{-2n}^{2c}\left( t\right) =c^{2}\left( \overline{J}%
_{-2n-2}^{2c}\left( t\right) -2\overline{J}_{-2n}^{2c}\left( t\right) +%
\overline{J}_{-2n+2}^{2c}\left( t\right) \right) \\
&=&c^{2}\left( u\left( n+1;t\right) -2u\left( n;t\right) +u\left(
n-1;t\right) \right) .
\end{eqnarray*}

For $n=0$, using Lemma \ref{lemma2} (ii) and (v), we obtain%
\begin{eqnarray*}
\overline{\partial }_{t}u\left( n;t\right) &=&\overline{\partial }_{t}%
\overline{J}_{0}^{2c}\left( t\right) =-2c\overline{J}_{1}^{2c}\left(
t\right) , \\
\overline{\partial }_{t}^{2}u\left( n;t\right) &=&\overline{\partial }%
_{t}^{2}\overline{J}_{0}^{2c}\left( t\right) =-2c\overline{\partial }_{t}%
\overline{J}_{1}^{2c}\left( t\right) =-2c^{2}\left( \overline{J}%
_{0}^{2c}\left( t\right) -\overline{J}_{2}^{2c}\left( t\right) \right) \\
&=&c^{2}\left( u\left( 1;t\right) -2u\left( 0;t\right) +u\left( -1;t\right)
\right) .
\end{eqnarray*}

Therefore, the equation \eqref{backWaveEq} holds for all $n\in \mathbb{Z}$
and $t\in \mathbb{N}_{0}$.
\end{proof}
\end{theorem}

The second fundamental solution to the backward time discrete wave equation is the solution $u_2(x;t)$, $x\in\mathbb{Z}$, $t\in\mathbb{N}_0$ to \eqref{backWaveEq} satisfying the initial conditions
\begin{equation}\label{eq. init cond u2}
u_{2}\left( n;0\right) =0,\quad \overline{\partial }_{t}u_2\left( n;0\right)
=\left\{
\begin{array}{ll}
1 & \text{if }n=0, \\
0 & \text{if }n\neq 0,%
\end{array}%
\right. \text{ \ }n\in \mathbb{Z}.
\end{equation}
We have the following proposition.

\begin{proposition}
Let $c>0$. The solution of the backward wave equation \eqref{backWaveEq}
with initial conditions \eqref{eq. init cond u2} is given by
\begin{equation}\label{eq. u2 defn}
u_{2}\left( n;t\right) =\sum_{s=0}^{t}\overline{J}_{2\left\vert n\right\vert
}^{2c}\left( s\right) - \overline{J}_{2\left\vert n\right\vert
}^{2c}\left( -1\right),
\end{equation}
where we assume that $t\in\mathbb{Z}_{\geq -1}$ and set the empty sum to be identically zero.
\end{proposition}
\begin{proof}
Let $t\geq 1$. By definition, $\overline{\partial}_t u_2(n;t) = \overline{J}_{2\left\vert n\right\vert
}^{2c}\left( t\right) $. Now, using part (v) of Lemma \ref{lemma2}, it is straightforward to check that  the function $u_{2}$ is the solution of the backward discrete wave equation \eqref{backWaveEq}.

It remains to check the initial conditions. When $t=0$ we have $u_2(n;0)= \overline{J}_{2\left\vert n\right\vert}^{2c}\left( 0\right) - J_{2\left\vert n\right\vert
}^{2c}\left( -1\right)$, where we used the identity $\overline{J}_n^c(-t)=(-1)^n J_n^c(t)$, for non-negative integers $t$. Both terms on the right-hand side are equal to zero for $|n|\geq 1$ and equal to one when $n=0$, which proves that $u_2(n;0)=0$. Finally, $\overline{\partial }_{t}u_2\left( n;0\right) = u_2(n;0) - u_2(n;-1)=\overline{J}_{2\left\vert n\right\vert}^{2c}\left( 0\right)$ which is zero unless $n=0$ in which case it equals one.
This proves that $u_2(n;t)$, given by \eqref{eq. u2 defn}, satisfies the initial conditions \eqref{eq. init cond u2} and completes the proof.
\end{proof}

\subsection{General solution to the discrete backward equation}

Now that we found two fundamental solutions to \eqref{backWaveEq}, we are in a position to find the solution to \eqref{backWaveEq} under general initial conditions given by arbitrary bounded real sequences indexed by integers.


\begin{theorem}
\label{thgensol} \bigskip For each $c>0$ and arbitrary bounded real
sequences $\left\{ u_{n}^{0}\right\} _{n\in \mathbb{Z}}$ and $\left\{
v_{n}^{0}\right\} _{n\in \mathbb{Z}}$, the general solution of the wave
equation \eqref{backWaveEq} with initial conditions
\begin{equation}
u\left( n;0\right) =u_{n}^{0},\text{ }\overline{\partial }_{t}u\left(
n;0\right) =v_{n}^{0},\text{ }n\in \mathbb{Z},  \label{genInitialCon}
\end{equation}%
is given by%
\begin{equation}  \label{gensolwave}
u\left( n;t\right) =\sum\limits_{k\in \mathbb{Z}}\left( u_{k}^{0}\cdot
u_{1}\left( n-k;t\right) +v_{k}^{0}\cdot u_{2}\left( n-k;t\right) \right) ,%
\text{ \ }n\in \mathbb{Z}\text{, \ }t\in \mathbb{N}_{0},
\end{equation}%
where $u_{1}\left( n;t\right) =\overline{J}_{2\left\vert n\right\vert
}^{2c}\left( t\right) $, and $u_{2}\left( n;t\right)
=\sum_{s=0}^{t}u_{1}\left( n;s\right) -u_{1}\left( n;-1\right) $.

\begin{proof}

From \cite[Theorem 2.5]{Slav17} with the backward difference as the timescale derivative, it follows that in order to prove that the function \eqref{gensolwave} is
the unique solution of the backward discrete wave equation \eqref{backWaveEq}
satisfying \eqref{genInitialCon} it suffices to prove that the series on the
right-hand side of \eqref{gensolwave} is absolutely convergent for all $t\in\mathbb{N}_0$. (The proof is analogous to the proof of \cite[Theorem 3.2]{Slav17}, so we omit it here.)

According to Proposition \ref
{asympJn}, for $t\in\mathbb{N}_0$, numbers $\left|\overline{J}_{2\left\vert n-k\right\vert
}^{2c}\left( t\right) \right|$ decay exponentially as $|k|\to\infty$, hence the series%
\begin{equation*}
\sum\limits_{k\in \mathbb{Z}}u_{k}^{0}\cdot \overline{J}_{2\left\vert
n-k\right\vert }^{2c}\left( t\right)
\end{equation*}%
is absolutely convergent for every bounded real sequence $\left\{ u_{n}^{0}\right\}
_{n\in \mathbb{Z}}$, fixed $t\in \mathbb{N}_{0}$ and $c>0$. Since $%
u_{2}\left( n;t\right) $ is a finite sum of functions $\overline{J}%
_{2\left\vert n-k\right\vert }^{2c}\left( t\right) $, using the same
argument as above, we conclude that the series%
\begin{equation*}
\sum\limits_{k\in \mathbb{Z}}v_{k}^{0}\cdot u_{2}\left( n-k;t\right)
\end{equation*}%
is also absolutely convergent for every bounded real sequence $\left\{
v_{n}^{0}\right\} _{n\in \mathbb{Z}}$, fixed $t\in \mathbb{N}_{0}$ and $c>0$%
. Therefore, the general solution to the backward discrete wave equation %
\eqref{backWaveEq} with initial conditions \eqref{genInitialCon} is given by %
\eqref{gensolwave}.
\end{proof}
\end{theorem}


\subsection{Asymptotic behaviour of solutions to discrete wave equations}

Applying the asymptotic formula \eqref{asymp1} of Theorem \ref{th_asymp} to the solution  \eqref{J2n} of the wave equation \eqref{discwaveq1} subject to the initial conditions \eqref{dweqcon1} we easily deduce the limiting behaviour of solutions  when the time
variable tends to infinity, as described in the following corollary.

\begin{corollary}\label{cor forward asympt}
For $c>0$, the solution $u(n;t)=J_{2|n|}^{2c}(t)$, $n\in\mathbb{Z}$, $t\in\mathbb{N}_0$ to the discrete wave equation \eqref{discwaveq1} subject to the initial conditions \eqref{dweqcon1} has the following asymptotic behavior
\begin{equation*}
u\left( n;t\right) \sim \frac{2}{\sqrt{\pi tc}}\left( 1+\frac{c^{2}}{4}%
\right) ^{\frac{t}{2}+\frac{1}{4}}\cos \left( \left( t+\frac{1}{2}\right)
\theta +\frac{\left\vert n\right\vert -1}{4}\pi \right), \text{ \ as }%
t\rightarrow \infty \text{.}
\end{equation*}
\end{corollary}

Oscillations with exponentially growing amplitude are somewhat unexpected, however, as seen in \cite{Ch22}, in some situations solutions to a discrete semilinear wave equation can blow up in finite time (on a continuous timescale).

From the Theorem \ref{th_asymp} we can derive the following asymptotic
behaviour of the solution to the backward discrete wave equation.

\begin{corollary}\label{cor backward asympt}
For $c>0$,  the solution $u\left( n;t\right) =\overline{J}_{2\left\vert n\right\vert }^{2c}\left(
t\right)$, $n\in\mathbb{Z}$, $t\in\mathbb{N}_0$ to the discrete wave
equation \eqref{backWaveEq} subject to the initial conditions %
\eqref{icWaveEq} has the following asymptotic behavior
\begin{equation*}
u\left( n;t\right) \sim \frac{2}{\sqrt{\pi tc}}\left( 1+\frac{c^{2}}{4}%
\right) ^{-\frac{t}{2}+\frac{1}{4}}\cos \left( \left( t-\frac{1}{2}\right)
\theta +\frac{\left\vert n\right\vert -1}{4}\pi \right), \text{ \ as }%
t\rightarrow \infty \text{.}
\end{equation*}
\end{corollary}


\begin{thebibliography}{99}

\bibitem{AS} M.~Abramowitz and A.~Irene Stegun, \emph{Handbook of
mathematical functions with formulas, graphs, and mathematical tables},
Tenth Edition, National Bureau of Standards Applied Mathematics Series, No.
55 U. S. Government Printing Office, Washington, D.C., 1972.

\bibitem{AMPS13} J.~-P.~Anker, P.~Martinot, E.~Pedon, and A.~G.~Setti, \emph{%
The shifted wave equation on Damek-Ricci spaces and on homogeneous trees},
Trends in harmonic analysis, 1--25, Springer INdAM Ser., 3, Springer, Milan,
2013.

\bibitem{BM01} J. Berkovits, J. Mawhin, J. \emph{Diophantine approximation, Bessel functions and radially symmetric periodic solutions of semilinear wave equations in a ball}, Trans. Amer. Math. Soc. \textbf{353} (2001), no. 12, 5041--5055.

\bibitem{BC} M.~Bohner and T.~Cuchta, \emph{The Bessel difference equation},
Proc. Am. Math. Soc. \textbf{145} (2017), no. 4, 1567--1580.


\bibitem{BG10} M.~Bohner and G.~Sh.~Guseinov, \emph{The Laplace transform on
isolated time scales}, Comput. Math. Appl. \textbf{60} (2010), no. 6,
1536--1547.

\bibitem{BP1} M.~Bohner and A.~Peterson, \emph{Dynamic Equations on Time
Scales: An Introduction with Applications}, Birkh\"{u}ser, Boston, 2001.

\bibitem{BPLapl} M.~Bohner and A.~Peterson, \emph{Laplace transform and
Z-transform: unification and extension}, Methods Appl. Anal. \textbf{9}
(2002), no. 1, 151--157.

\bibitem{BP2} M.~Bohner and A.~Peterson, \emph{Advances in Dynamic Equations
on Time Scales}, Birkh\"{u}ser, Boston, 2003.

\bibitem{Boy61} R. H. Boyer, \emph{Discrete Bessel functions}, J. Math.
Anal. Appl. \textbf{2} (1961), 509--524.

\bibitem{CHJSV} C.~A. Cadavid, P.~Hoyos, J.~Jorgenson, L.~Smajlovi\'{c}, and
J.~D.~V\'{e}lez, \emph{Discrete diffusion-type equation on regular graphs
and its applications}, J. Difference Equ. Appl. \textbf{29} (2023), no. 4, 455--488.

\bibitem{Ch22} M.-J. Choi, \emph{A condition for blow-up solutions to discrete semilinear wave equations on networks}, Appl. Anal. \textbf{101} (2022), no. 6, 2008--2018.

\bibitem{CJK14} G.~Chinta, J.~Jorgenson, and A.~Karlsson, \emph{Heat kernels
on regular graphs and generalized Ihara zeta function formulas}, Monatsh.
Math. \textbf{178} (2015), no. 2, 171--190.

\bibitem{Cu15} T.~Cuchta, \emph{Discrete analogues of some classical special
functions} (2015). Doctoral Dissertation, Missouri University of Science and
Technology.

\bibitem{CGW23} T.~Cuchta, D.~Grow, and N.~Wintz, \emph{Discrete matrix
hypergeometric functions}, J. Math. Anal. Appl. \textbf{518} (2023), no. 2,
Paper No. 126716, 14 pp.

\bibitem{CP94} J.~M. Cohen and M.~Pagliacci, \emph{Explicit solutions for
the wave equation on homogeneous trees}, Adv. in Appl. Math. \textbf{15}
(1994), no. 4, 390--403.


\bibitem{DGJMR07} J. M. Davis, I. A. Gravagne, B. J. Jackson, R. J. Marks, A. A. Ramos, \emph{The Laplace transform on time scales revisited}, J. Math. Anal. Appl. \textbf{332} (2007), no. 2, 1291--1307.

\bibitem{G-CKLW19} J. Gonz\' alez-Camus, V. Keyantuo, C. Lizama, M. Warma, \emph{Fundamental solutions for discrete dynamical systems involving the fractional Laplacian}, Math. Methods Appl. Sci. \textbf{42} (2019), no. 14, 4688--4711.

\bibitem{GR} I.~S.~Gradschteyn and I.~M.~Ryzhik, \emph{Table of integrals,
series, and products}, Translated from the Russian. Translation Edited and
with a Preface by Alan Jeffrey and Daniel Zwillinger. With one CD-ROM
(Windows, Macintosh and UNIX), Seventh edition, Elsevier/Academic Press,
Amsterdam, 2007.

\bibitem{Ja06} B.~Jackson, \emph{Partial dynamic equations on time scales},
J. Comput. Appl. Math. \textbf{186} (2006), no. 2, 391--415.

\bibitem{JD21} B. J. Jackson, J. M. Davis, \emph{An ergodic approach to Laplace transforms on time scales},
J. Math. Anal. Appl. \textbf{502} (2021), 125231 (31 pp).

\bibitem{KS22} N.~Kan and K.~Shiraishi, \emph{Doscrete time heat kernel and
UV modified propagators with Dimensional Deconstruction}, Journal of Physics A: Mathematical and Theoretical, \textbf{56}  No. 24 (2023), 245401 (16pp).


\bibitem{Leb} N.~N.~Lebedev, \emph{Special functions and their applications}%
, Revised English edition. Translated and edited by Richard A. Silverman.
Prentice-Hall, Inc., Englewood Cliffs, N.J., 1965.

\bibitem{LL61} H.~Levy and F.~Lessman, \emph{Finite difference equations},
The Macmillan Company, New York, 1961.

\bibitem{LM23}  C. Lizama, M. Murillo-Arcila, \emph{The semidiscrete damped wave equation with a fractional Laplacian}, Proc. Amer. Math. Soc. \textbf{151} (2023), no. 5, 1987--1999.

\bibitem{Ma13} M.~Mansour, \emph{Generalized q-Bessel function and its
properties}, Adv. Difference Equ. \textbf{2013}, 2013:121, 11 pp.

\bibitem{MS99} G.~Medolla and A.~G.~Setti, \emph{The wave equation on
homogeneous trees}, Ann. Mat. Pura Appl. (4) \textbf{176} (1999), 1--27.

\bibitem{Me99} G.~Medolla, \emph{Asymptotic energy equipartition for the
wave equation on homogeneous trees}, Monatsh. Math. \textbf{127} (1999), no.
1, 43--53.

\bibitem{OLBC} F.~W.~J.~Olver, D.~W.~Lozier, R.~F.~Boisvert, and C.~W.~Clark
(eds.), NIST Handbook of Mathematical Functions, Cambridge University Press,
New York, 2010. Available at http://dlmf.nist.gov/.

\bibitem{RKN19} M.~Riyasat, S.~Khan, and T.~Nahid, \emph{Quantum algebra $%
\mathcal{E}_q(2)$ and $2D$ $q$-Bessel functions}, Rep. Math. Phys. \textbf{83%
} (2019), no. 2, 191--206.


\bibitem{Slav17} A.~Slav\'{\i}k, \emph{Discrete-space systems of partial
dynamic equations and discrete-space wave equation}, Qual. Theory Dyn. Syst.
\textbf{16} (2017), no. 2, 299--315.

\bibitem{Slav18} A.~Slav\'{\i}k, \emph{Discrete Bessel functions and partial
difference equations}, J. Difference Equ. Appl. \textbf{24} (2018), no. 3,
425--437.

\bibitem{Slav22} A.~Slav\'{\i}k, \emph{Asymptotic behavior of solutions to
the multidimensional semidiscrete diffusion equation}, Electron. J. Qual.
Theory Differ. Equ. 2022, Paper No. 9, 9 pp.

\bibitem{Sw92} R.~F.~Swarttouw, \emph{The Hahn-Exton $q$-Bessel function},
Ph. D. Thesis, Technische Universiteit Delft, 1992, 89 pp.

\bibitem{TS16} A.~V.~Tsvetkova and A.~I.~Shafarevich, \emph{The Cauchy
problem for the wave equation on a homogeneous tree}, Mat. Zametki \textbf{%
100} (2016), no. 6, 923--931; translation in Math. Notes \textbf{100}
(2016), no. 5--6, 862--869.

\bibitem{Th} R.~C.~Thorne, \emph{The Asymptotic Expansion of Legendre
Functions of Large Degree and Order}, Philos. Trans. Roy. Soc. London Ser. A
\textbf{249} (1957), 597--620.

\bibitem{YYT22} A.~Yant{\i}r, B.~Silindir Yant{\i}r, and Z.~Tuncer, \emph{%
Bessel equation and Bessel function on $\mathbb{T}_{(q,h)}$}, Turkish J.
Math. \textbf{46} (2022), no. 8, 3300--3322.

\bibitem{Watson} G.~N.~Watson, \emph{Asymptotic expansions of hypergeometric
functions}, Trans. Cambridge Philos. Soc. \textbf{22} (1918) 277--308.
\end{thebibliography}
\end{document}